\documentclass{amsart}

\usepackage{amsmath}
\usepackage{amsxtra}
\usepackage{amscd}
\usepackage{amsthm}
\usepackage{amsfonts}
\usepackage{amssymb}
\usepackage{eucal}
\usepackage[all]{xy}
\usepackage{graphicx}
\usepackage{tikz-cd}
\usepackage{mathrsfs}
\usepackage{euler}
\usepackage{hyperref}
\usepackage{color}
\usepackage{longtable}
\usepackage{float}
\usepackage{caption}
\usepackage{enumerate}
\usepackage{faktor}

\usepackage{tikz}
\usepackage{tikz-cd}
\usetikzlibrary{matrix,arrows,decorations.pathmorphing}

\theoremstyle{plain}
  \newtheorem{theorem}{Theorem}[section]
	\newtheorem{thm}[theorem]{Theorem}
  
  \newtheorem{prop}[theorem]{Proposition}
  \newtheorem{lemma}[theorem]{Lemma}
  \newtheorem{lem}[theorem]{Lemma}
  \newtheorem{corollary}[theorem]{Corollary}
  \newtheorem{cor}[theorem]{Corollary}

\theoremstyle{definition}
	\newtheorem{definition}[theorem]{Definition}
	\newtheorem{defn}[theorem]{Definition}

	\newtheorem{non-example}[theorem]{Non-example}

\theoremstyle{remark}
	\newtheorem{remark}[theorem]{Remark}
	\newtheorem{rem}[theorem]{Remark}
	
	\newtheorem{question}[theorem]{Question}

\numberwithin{equation}{section}

\DeclareMathAlphabet{\mathcal}{OMS}{cmsy}{m}{n}

\newcommand{\C}{\mathbb{C}}

\newcommand{\F}{\mathbb{F}}

\newcommand{\N}{\mathbb{N}}

\newcommand{\Q}{\mathbb{Q}}

\newcommand{\sco}{\mathscr{O}}

\newcommand{\calf}{\mathcal{F}}
\newcommand{\calg}{\mathcal{G}}
\newcommand{\calh}{\mathcal{H}}
\newcommand{\calo}{\mathcal{O}}

\newcommand{\mfm}{\mathfrak{m}}

\newcommand{\Char}{\mathrm{char}\,}
\newcommand{\im}{\mathrm{im}}
\newcommand{\coker}{\mathrm{coker}}

\newcommand{\tr}{\mathrm{tr}}

\renewcommand{\hom}{\mathrm{Hom}}

\newcommand{\res}{\mathrm{Res}}
\newcommand{\ind}{\mathrm{Ind}}
\newcommand{\coind}{\mathrm{CoInd}}

\newcommand{\cind}{\mathrm{c}\text{--}\mathrm{Ind}}
\newcommand{\lind}{\mathrm{ind}}
\newcommand{\End}{\mathrm{End}}
\newcommand{\rad}{\mathrm{rad}}
\newcommand{\soc}{\mathrm{soc}}
\newcommand{\cosoc}{\mathrm{cosoc}}

\newcommand{\mat}{\mathrm{Mat}}
\newcommand{\gl}{\mathrm{GL}}

\newcommand{\gu}{\mathrm{GU}}

\newcommand{\Sp}{\mathrm{Sp}}

\newcommand{\rep}{\mathrm{Rep}}

\newcommand{\triv}{\mathrm{triv}}

\newcommand{\qda}{D_k^*}

\newcommand{\gtone}{GT1}
\newcommand{\gpone}{GP1}
\newcommand{\gptwo}{GP2}
\newcommand{\gpthree}{GP3}

\newcommand{\rmod}[1]{#1\mathbf{-Mod}}

\newcommand{\bs}{\backslash}
\newcommand{\paren}[1]{\mathopen{}\left(#1\right)\mathclose{}}
\newcommand{\set}[1]{\mathopen{}\left\{#1\right\}\mathclose{}}

\newcommand{\verts}[1]{\mathopen{}\left\lvert#1\right\rvert\mathclose{}}

\newcommand\restr[2]{{
  \left.\kern-\nulldelimiterspace 
  #1 
  \right|_{#2} 
  }}

\title{Modular Gelfand pairs and multiplicity-free representations}

\author{Robin Zhang}
\address[Robin Zhang]{Department of Mathematics, Columbia University, Room 509, MC 4406, 2990 Broadway, New York, NY 10027 \newline
	\indent Department of Mathematics, Massachusetts Institute of Technology,
	Room 106, Simons Building (Building 2), 77 Massachusetts Avenue, Cambridge, MA 02139}
\email{rzhang@math.columbia.edu, robinz@mit.edu}

\date{August 3, 2023}

\begin{document}

\begin{abstract}
	We give a generalization of Gelfand's criterion
	on the commutativity of Hecke algebras for
	Gelfand pairs and multiplicity-free triples
	over algebraically closed fields of arbitrary
	characteristic.
	Using more lenient versions of projectivity
	and injectivity for modules, we prove a general
	multiplicity-freeness theorem for finitely-generated
	modules with commutative endomorphism rings.
	For representations of finite and profinite groups,
	Gelfand pairs over the complex numbers are therefore also
	Gelfand pairs over the algebraic closure of any finite field.
	Applications include the uniqueness of Whittaker models of
	modular Gelfand--Graev representations and the uniqueness
	of modular trilinear forms on irreducible representations
	of quaternion division algebras over local fields.
\end{abstract}


\maketitle

\setcounter{tocdepth}{1}

\tableofcontents


\section{Introduction}

For representations of a group over the complex numbers, the
classical theory of Gelfand pairs
considers group--subgroup
pairs $(G, H)$ such that the induced representation
$\ind_H^G(\triv_H) = \C[G] \otimes_{\C[H]} \triv_H$ of the trivial
representation of the subgroup $H$ is multiplicity-free.
For finite groups, Gelfand pairs arise as useful tools
across disciplines of mathematics, such as in number theory
(cf. \cite{gross}),
combinatorics (cf. \cite{aker-can}, \cite[Chapter~4]{johnson},
\cite[Section~3]{martin-tanaka}), and probability theory
(cf. \cite{diaconis-1988, diaconis-1996, ceccherini,
ceccherini-2}).
The fact that $(S_n, S_{n-1})$ is a Gelfand pair over $\C$, for
instance, is important in the study of the Gelfand--Tsetlin algebra in
the representation theory of $S_n$
(cf. \cite[Section~2]{vershik-okounkov, vershik-okounkov-2}).
The uniqueness of Whittaker models is a fundamental result
in the theory of automorphic representations; for the finite group
$\gl_n\paren{\F_{q}}$,
the key idea of Gelfand--Graev \cite{gelfand-graev}
is that $\paren{\gl_n\paren{\F_{q}}, N}$ is a (twisted)
Gelfand pair for the standard nilpotent subgroup $N$
of $\gl_n\paren{\F_{q}}$.

Due to its broad applications,
the concept of Gelfand pairs has been generalized in several
directions from its original form
for complex representations
of finite groups and Lie groups with compact subgroups.
In particular, the theory of Gelfand pairs has been
fruitfully extended by replacing the induced trivial representation
$\ind_H^G(\triv_H)$ with $\ind_H^G(\eta)$
for an arbitrary irreducible representation $\eta$ of $H$.
This generalization is the study of \textit{Gelfand triples}
or \textit{multiplicity-free triples} $(G, H, \eta)$
(cf. Bump \cite[Section~45]{bump-lie},
Ceccherini-Silberstein--Scarabotti--Tolli \cite{csst}). Since
there is an overlap of terminology with the usage of
``Gelfand triple'' as a synonym for a
rigged Hilbert space in functional analysis, this article will use the
less ambiguous ``multiplicity-free triple''.

The property of ``multiplicity-freeness'' for
a representation $\rho$ of a group $G$ over a field
$F$ here means that irreducible representations occur
in $\rho$ at most once, i.e. that for all irreducible representations
$\pi$ of $G$ over $F$,
\[
	\dim_F \hom_H(\pi, \rho) \leq 1.
\]

\begin{rem}
	Note that for non-semisimple categories of representations,
	this differs from the typical notion of multiplicity-freeness that is
	defined in terms of unique composition factors; in that sense,
	this is the multiplicity-freeness of the socle.
	This property can also be relative to the chosen category of
	representations and allow for additional conditions,
	such as requiring that the
	inequality holds only for smooth or admissible
	irreducible representations $\pi$ of $G$ over $F$.
	For much of this article, we will restrict ourselves to
	algebraically closed $F$,
	finite or compact $G$,
	and category of representations $\rep_F(G)$.
	When no topology is specified for infinite groups and
	their representation spaces, we generally
	assume the discrete topology.
\end{rem}

Following the (\gpone, \gptwo, \gpthree)
convention of Aizenbud--Gourevitch--Sayag
\cite[Definition~2.2.1]{ags},
the following is a definition of a multiplicity-free triple
for finite and compact groups over algebraically closed
fields.
\begin{defn}[Multiplicity-free triple, \gtone]
	\label{def:gt1}
	Let $F$ be an algebraically closed field, $G$ be a
	finite or compact group, $H$ be a subgroup of $G$,
	and $\eta$ be a representation of $H$.
	The triple $(G, H, \eta)$ is a multiplicity-free triple
	over $F$ if
	\[
		\dim_F \hom_G\paren{\pi, \ind_H^G\paren{\eta}} \leq 1.
	\]
	for all irreducible representations $\pi$ of $G$.
\end{defn}

\begin{rem}
	More generally when $F$ is not algebraically closed,
	we can modify Definition~\ref{def:gt1} to
	ask instead that
	$\hom_G \paren{\pi, \ind_H^G(\eta)}$
	be a free $\End_G(\pi)$-module of rank $\leq 1$.
	Over algebraically closed $F$,
	this rank $\leq 1$ condition is equivalent
	to the dimension condition of Definition~\ref{def:gt1}
	since the right-hand side is an $\End_G(\pi)$-module,
	$\End_G(\pi)$ is a division algebra if $\pi$ is
	irreducible, and $\End_G(\pi) = F$ if and only if
	$\pi$ is absolutely irreducible.
	Furthermore, 
	$(G, H, \triv_H)$ being a multiplicity-free triple (\gtone)
	is equivalent to $(G, H)$ being a Gelfand pair (\gpone)
	when $F$ is algebraically closed.
	Here, moving from the condition of Definition~\ref{def:gt1} to
	the definition of Gelfand pairs (\gpone) of 
	Aizenbud--Gourevitch--Sayag \cite[Definition~2.2.1]{ags}
	is due to Frobenius reciprocity,
	\[
		\hom_H\paren{\res_{H}^{G} \paren{\pi}, \triv_H} \cong
		\hom_G\paren{\pi, \ind_H^G\paren{\triv_H}}.
	\]
\end{rem}

Classically, the property of being a Gelfand pair
is related to the commutativity of the Hecke algebra of $G$ with
respect to $H$ (with $\eta = \triv_H$).
\begin{defn}
	For a compact group $G$, closed subgroup $H$ of $G$,
	and representation $\eta$ of $H$,
	the Hecke algebra $\calh(G, H, \eta, F)$ over $F$, also written as
	$\calh(G//H, \eta, F)$, is the convolution algebra of continuous
	functions $\Delta: G \rightarrow \End_F(\eta)$ satisfying
	\[
		\Delta(h_2 g h_1) = \eta(h_2) \circ \Delta(g) \circ \eta(h_1),
	\]
	for all $g \in G$ and $h_1, h_2 \in H$.
\end{defn}

\noindent
This Hecke algebra is isomorphic to
$\End_G(\coind_H^G\paren{\eta})$ by Mackey theory
(cf. \cite[Theorem~45.1]{bump-lie}),
where coinduction $\coind_H^G(-)$ is $\hom_{F[H]}(F[G], -)$.
When $G$ is a finite group with a subgroup $H$,
there is the following well-known
criterion, essentially due to
Gelfand \cite{gelfand} and Gelfand--Graev \cite{gelfand-graev},
for $(G, H)$ to be a Gelfand pair.
\begin{prop}[{\cite{gelfand, gelfand-graev}}]
\label{prop:classical}
	Let $G$ be a finite group, $H$ be a subgroup of $G$,
	and $\triv_H$ be the trivial representation of $H$.
	If the Hecke algebra $\calh(G, H, \triv_H, \C)$
	is commutative, then
	$(G, H, \triv_H)$ is a multiplicity-free triple over $\C$.
\end{prop}
The classical proof of Proposition~\ref{prop:classical} can
be trivially generalized to obtain the same result 
over any algebraically closed field $F$
of characteristic $\ell$ not dividing $\verts{G}$
and for any irreducible representation $\eta$ of $H$.
However, the classical proof relies on Maschke's theorem and
Schur's lemma and therefore does not work for fields of characteristic
dividing $\verts{G}$ nor for fields that are not algebraically closed.
If the category of representations is semisimple, then
one can get around using Schur's lemma (and therefore
algebraic closure) because
\[
	\End_G\paren{\ind_H^G(\triv_H)} \cong F[H \backslash G / H],
\]
but the obstruction remains when $\rep_F(G)$
is not semisimple.

The motivation of this article is to extend
Gelfand's criterion (Proposition~\ref{prop:classical}) to
algebraically closed fields $F$ of arbitrary characteristic.

\begin{thm}[Theorem~\ref{thm:gelfand-triple-2}]
	\label{thm:gelfand-triple}
	Let $F$ be an algebraically closed field, $G$ be a finite group,
	$H$ be a subgroup of $G$,
	and $\eta$ be an irreducible representation of $H$.
	If $\calh(G, H, \eta, F)$ is commutative, then $(G, H, \eta)$ is a
	multiplicity-free triple over $F$.
\end{thm}

\noindent
To prove this generalization of Gelfand's criterion, we establish
generalities about relative-projectivity and relative-injectivity
(in the sense of Sandomierski and Azumaya), which are more lenient 
versions of usual projectivity and injectivity. The main input of
this article is the following general multiplicity-freeness theorem
(stated for more general modules in
Theorem~\ref{thm:multiplicity-free-2}).

\begin{theorem}[Theorem~\ref{thm:multiplicity-free-3}]
\label{thm:multiplicity-free}
	Let $F$ be an algebraically closed field and
	$G$ be a finitely-generated group, and
	$\rho$ be a finite-dimensional
	representation of $G$. Then $\rho$ is multiplicity-free
	if both of the following conditions are satisfied:
	\begin{enumerate}[(i)]
		\item $\End_G(\rho)$ is commutative;
		\item $\rho$ is a self-injective $F[G]$-module.
	\end{enumerate}
\end{theorem}

In the finite group
applications of Theorem~\ref{thm:gelfand-triple}
that we highlight, we use the following consequence
when $F$ is the algebraic closure of a finite field.
\begin{corollary}[Corollary~\ref{cor:Fq-2}]
	\label{cor:Fq}
	Let $F = \overline{\F}_\ell$ of any positive characteristic $\ell$,
	$G$ be a finite group, $H$ be a subgroup of $G$, and
	$\eta$ be a complex multiplicative character of $H$.
	If $(G, H, \eta)$ is a multiplicity-free triple over $\C$,
	then $(G, H, \overline{\eta})$ is also a multiplicity-free triple
	over $F$.
\end{corollary}
\noindent
As a consequence, the
classical multiplicity-freeness theorems for complex representations
of finite groups also hold over $F$. This extends to multiplicity-free
triples with characters $\eta$ and to
totally disconnected compact groups as well.
For example,
Corollary~\ref{cor:Fq} yields the
multiplicity freeness of induced characters of the unipotent subgroup
in the theory of Gelfand--Graev representations.
By Corollary~\ref{cor:Fq},
we immediately obtain the generalization
(already known due to the work of 
Curtis \cite{curtis1965, curtis1970}, Richen \cite{richen}, and
Steinberg \cite{steinberglectures})
of the classical Gelfand--Graev \cite{gelfand-graev}
multiplicity-one theorem without having to use the
theory of highest weights and Tits buildings.

\begin{cor}[{Corollary~\ref{cor:whittaker-2}}]
	\label{cor:whittaker}
	Let $F = \overline{\F}_\ell$ of any positive characteristic $\ell$
	and let $K$ be any finite field $\F_q$.
	An irreducible representation of $\gl_n(K)$
	over $F$ has at most one Whittaker model.
\end{cor}

Corollary~\ref{cor:Fq} can also be used for
certain infinite groups. Since irreducible smooth representations of
totally disconnected compact (i.e. profinite) groups factor
through finite quotients,
we can demonstrate a multiplicity-one theorem for
modular trilinear forms of
quaternion division algebras, whose complex version
is of arithmetic interest
due to its role in the non-vanishing of triple product
L-functions and the development of the Gan--Gross--Prasad conjecture.

\begin{cor}[Corollary~{\ref{cor:qda-2}}]
	\label{cor:qda}
	Let $F = \overline{\F}_\ell$ of any positive characteristic $\ell$,
	$k$ be any local field, and $D_k$ be the quaternion division
	algebra over $k$.
	If $V_1, V_2, V_3$ are three irreducible smooth
	representations of $\qda$ over $F$, then there exists at most one
	(up to isomorphism) non-zero $\qda$-invariant linear form
	on $V_1 \otimes V_2 \otimes V_3$ over $F$.
\end{cor}

\begin{remark}
	The criterion given in Theorem~\ref{thm:multiplicity-free}
	is stated in sufficient generality for non-compact groups.
	One area with automorphic applications for future study are
	totally disconnected locally compact groups; in this setting
	the Gelfand--Kazhdan criterion for
	(twisted) Gelfand pairs is
	also the commutativity of Hecke algebras
	(cf. \cite{gelfand-kazhdan, hendel}).
	Multiplicity-free triples for such reductive groups
	are important objects in areas like 
	the Gan--Gross--Prasad conjecture where
	\textit{Gan--Gross--Prasad triples} are known in certain cases to be
	multiplicity-free over the complex numbers
	(cf. Beuzart--Plessis \cite{bp} and Luo \cite{luo}).
	Another potentially interesting direction is to modify
	Theorem~\ref{thm:gelfand-triple} for cuspidal Gelfand pairs
	(cf. \cite{baruch-rallis}).
\end{remark}

Finally, we explore a
multiplicity-freeness question about restrictions
of representations of groups that was considered in
characteristic $0$ by Weyl \cite{weyl},
Kac \cite{kac}, Howe \cite{howe}, Stembridge \cite{stembridge},
and more recently by Liebeck--Seitz--Testerman \cite{lst1, lst2}.
One consequence of Theorem~\ref{thm:multiplicity-free}
is a Gelfand-like criterion for multiplicity-free restrictions.

\begin{cor}[{Corollary~\ref{cor:restriction}}]
	\label{cor:restriction-2}
	Let $F$ be an algebraically closed field,
	$H$ be a subgroup of a discrete finitely generated group $G$ such that
	$\ind_H^G$ is an exact functor, and $\rho$ be a finite-dimensional
	representation of $G$.
	If $\rho$ is $\ind_H^G(\res_H^G(\rho))$-injective and
	$\End_H(\res_H^G(\rho))$ is commutative, then
	$\res^{G}_{H}(\rho)$ is multiplicity-free.
\end{cor}


\section{Relative projectivity and relative injectivity}
\label{sec:rel-proj}

We recall some useful characterizations about relatively-projective
modules and relatively-injective
in the sense of Sandomierski and Azumaya, largely following
Wisbauer \cite[Chapter 3]{wisbauer} (cf. Azumaya \cite{azumaya},
Azumaya--Mbuntum--Varadarajan \cite{amv},
Elliger \cite{elliger}, and
Shrikhande \cite{shrikhande}).
In this section, $R$ denotes an
arbitrary ring with unity.

\subsection{Relatively-projective modules}

\begin{defn}
	Let $M$ and $N$ be $R$-modules.
	$M$ is called \emph{$N$-projective} if and only if
	for every $R$-module $K$, $R$-module
	homomorphism $f:M \rightarrow K$, and surjective $R$-module
	homomorphism $g:N \twoheadrightarrow K$, there is an $R$-module
	homomorphism
	$h:M \rightarrow N$ such that the following diagram commutes:
	\[
		\begin{tikzcd}
			& N \arrow[d, twoheadrightarrow, "g"] \\
			M \arrow[r, "f"] \arrow[ur, dashed, "h"] & K
		\end{tikzcd}
	\]
	
	\noindent
	An $R$-module $M$ is called \emph{self-projective} if it is
	$M$-projective.
\end{defn}

\begin{rem}
	What we have defined is an $R$-module being projective relative to
	another $R$-module. This notion of relative-projectivity is
	essentially the same as the
	one defined by Sandomierski \cite{sandomierski}
	(cf. de Robert \cite{derobert}), but is slightly different
	from the notion of relative-projectivity defined by
	Okuyama \cite{okuyama, carlson}.
	This is also distinct from the notion
	an $F[G]$-module being projective relative to a subgroup of $G$.
\end{rem}

An $R$-module $M$ is projective if and only if
$\hom(M, -)$ is an exact functor. Relative projectivity
can also be characterized with a similar exactness
condition. 
\begin{prop}
	\label{prop:rel-proj-exact}
	Let $M$ and $N$ be $R$-modules.
	$M$ is $N$-projective if and only if
	for every exact sequence
	\[
		\begin{tikzcd}
			0 \arrow[r] & L \arrow[r] & N \arrow[r] & K \arrow[r] & 0,
		\end{tikzcd}
	\]
	the corresponding sequence of $R$-module homomorphism groups is also
	exact:
	\[
		\begin{tikzcd}
			0 \arrow[r] & \hom_R(M, L) \arrow[r] & \hom_R(M, N) \arrow[r]
				& \hom_R(M, K) \arrow[r] & 0.
		\end{tikzcd}
	\]
\end{prop}

Note that the middle term $N$ of the exact sequence in
Proposition~\ref{prop:rel-proj-exact} is fixed
\textit{a priori}.
If we let $K = M$, then
Proposition~\ref{prop:rel-proj-exact} has the following
immediate corollary.
\begin{cor}
	Let $M$ and $N$ be $R$-modules.
	If $M$ is $N$-projective then every short exact
	sequence
	\[
		\begin{tikzcd}
			0 \arrow[r] & L \arrow[r] & N \arrow[r, "g"] & M \arrow[r] & 0,
		\end{tikzcd}
	\]
	splits.
\end{cor}

\begin{cor}
	Let $M$ and $N$ be $R$-modules with a surjective
	homomorphism $g: N \twoheadrightarrow M$.
	If $M$ is $N$-projective, then $M$ is a direct summand of $N$.	
	Furthermore, if $M$ is $N$-projective and $N$ is projective,
	then $M$ is projective.
\end{cor}

\begin{proof}
	The surjection $g$ induces a short exact sequence
	\[
		\begin{tikzcd}
			0 \arrow[r] & \ker(g) \arrow[r] & N
				\arrow[r, "g"] & M \arrow[r] & 0.
		\end{tikzcd}
	\]
	Since $M$ is $N$-projective, this
	short exact sequence splits so
	\[
		M \oplus \ker(g) \cong N.
	\]
	Therefore, $M$ is a direct summand of $N$.
	
	If $N$ is projective then it is the direct summand of a free module.
	Since $M$ is a direct summand of $N$, $M$ is also the direct
	summand of a free module and therefore projective.
\end{proof}

\begin{rem}
	The converse of the first claim is false. If
	\[
		\begin{tikzcd}
			& N \arrow[d, twoheadrightarrow, "g"] \\
			M \arrow[r, "f"] & K
		\end{tikzcd}
	\]
	is a diagram for which the lifting property fails, then
	\[
		\begin{tikzcd}
			& M \oplus N \arrow[d, twoheadrightarrow, "g'"] \\
			M \arrow[r, "f"] & K
		\end{tikzcd}
	\]
	with $g'(\ell, n) := g(n)$ is also a diagram
	for which the lifting property fails. So for any $R$-module $N$ such
	that $M$ is not $N$-projective, $M$ is also not
	$(M \oplus N)$-projective.
\end{rem}

\begin{definition}
	\hfill
	\begin{enumerate}[(i)]
		\item For an $R$-module $M$, let $C^p(M)$ denote the class of
			$R$-modules $N$ such that $M$ is $N$-projective.
		\item For an $R$-module $N$, let $C_p(N)$ denote the class of
			$R$-modules $M$ such that $M$ is $N$-projective.
	\end{enumerate}
\end{definition}

The following general properties of relative projectivity
are largely due to Azumaya \cite{azumaya} and
Shrikande \cite{shrikhande}.
\begin{prop}[{\cite[Proposition~1.16]{amv}, \cite[Section~18.2]{wisbauer}}]
	\label{prop:rel-proj-properties}
	\hfill
	\begin{enumerate}[(i)]
		\item For an $R$-module $M$, the class $C^p(M)$ is closed under
			taking submodules, finite direct sums, and
			images of $R$-module homomorphisms.
		\item For a finitely generated $R$-module $M$, $C^p(M)$ is also
			closed under arbitrary direct sums.
		\item For an $R$-module $N$, the class $C_p(N)$ is closed under
			taking arbitrary direct sums and direct summands.
	\end{enumerate}
\end{prop}

\subsection{Relatively-injective modules}

Similarly, there is the dual notion of relative injectivity
in the sense of Sandomierski and Azumaya.

\begin{defn}
	Let $M$ and $N$ be $R$-modules.
	$M$ is called \emph{$N$-injective} if and only if
	for every $R$-module $K$, injective $R$-module
	homomorphism $f:K \hookrightarrow N$, and $R$-module
	homomorphism $g:K \rightarrow M$, there is an $R$-module
	homomorphism
	$h:N \rightarrow M$ such that the following diagram commutes:
	\[
		\begin{tikzcd}
			K \arrow[r, hookrightarrow, "f"] \arrow[d, "g"] & N \arrow[dl, dashed, "h"] \\
			M & 
		\end{tikzcd}
	\]
	
	\noindent
	An $R$-module $M$ is called \emph{self-injective} if it is
	$M$-injective.
\end{defn}

An $R$-module $M$ is injective if and only if
$\hom(-, M)$ is an exact functor. Relative injectivity
can also be characterized with a similar exactness
condition. 
\begin{prop}
	\label{prop:rel-inj-exact}
	Let $M$ and $N$ be $R$-modules.
	$M$ is $N$-injective if
	for every exact sequence
	\[
		\begin{tikzcd}
			0 \arrow[r] & K \arrow[r] & N \arrow[r] & L \arrow[r] & 0,
		\end{tikzcd}
	\]
	the corresponding sequence of $R$-module homomorphism groups is also
	exact:
	\[
		\begin{tikzcd}
			0 \arrow[r] & \hom_R(K, M) \arrow[r] & \hom_R(N, M) \arrow[r]
				& \hom_R(L, M) \arrow[r] & 0.
		\end{tikzcd}
	\]
\end{prop}

As with Proposition~\ref{prop:rel-proj-exact} for relative
projectivity, the middle term $N$ of the exact
sequence in Proposition~\ref{prop:rel-inj-exact} is fixed
\textit{a priori}.
If we let $K = M$, then
Proposition~\ref{prop:rel-inj-exact} has the following
immediate corollary.
\begin{cor}
	Let $M$ and $N$ be $R$-modules.
	If $M$ is $N$-injective then every short exact
	sequence
	\[
		\begin{tikzcd}
			0 \arrow[r] & M \arrow[r, "f"] & N \arrow[r] & L \arrow[r] & 0,
		\end{tikzcd}
	\]
	splits.
\end{cor}

\begin{cor}
	Let $M$ and $N$ be $R$-modules with an injective
	homomorphism $f: M \hookrightarrow N$.
	If $M$ is $N$-injective, then $M$ is a direct summand of $N$.
\end{cor}

\begin{proof}
	The surjection $g$ induces a short exact sequence
	\[
		\begin{tikzcd}
			0 \arrow[r] & M \arrow[r, "f"] & N
				\arrow[r] & \coker(f) \arrow[r] & 0.
		\end{tikzcd}
	\]
	Since $M$ is $N$-projective, this
	short exact sequence splits so
	\[
		M \oplus \coker(f) \cong N.
	\]
	Therefore, $M$ is a direct summand of $N$.
\end{proof}

\begin{definition}
	\hfill
	\begin{enumerate}[(i)]
		\item For an $R$-module $M$, let $C^i(M)$ denote the class of
			$R$-modules $N$ such that $M$ is $N$-injective.
		\item For an $R$-module $N$, let $C_i(N)$ denote the class of
			$R$-modules $M$ such that $M$ is $N$-injective.
	\end{enumerate}
\end{definition}

The following general properties of relative injectivity
are largely due to Azumaya \cite{azumaya} and
Shrikande \cite{shrikhande}.
\begin{prop}[{\cite[Proposition~1.16]{amv}, \cite[Section~16.2]{wisbauer}}]
	\label{prop:rel-inj-properties}
	\hfill
	\begin{enumerate}[(i)]
		\item For an $R$-module $M$, the class $C^i(M)$ is closed under
			taking submodules, arbitrary direct sums, and
			images of $R$-module homomorphisms.
		\item For a finitely generated $R$-module $M$, $C^i(M)$ is also
			closed under arbitrary direct sums.
		\item For an $R$-module $N$, the class $C_i(N)$ is closed under
			taking arbitrary direct products and direct factors.
	\end{enumerate}
\end{prop}

\subsection{Endomorphism rings}

We recall some more definitions from module theory.
Let $M$ be an $R$-module. An $R$-submodule $N$ of $M$ is
called \emph{maximal} if $M/N$ is a simple $R$-module.
The \emph{radical} $\rad(M)$ is defined to be the intersection of all
maximal submodules of $M$. The \emph{socle} $\soc(M)$ is defined to be
the maximal semisimple submodule of $M$, and is equal to the
sum of all simple submodules of $M$.
The \emph{cosocle} (also called the \emph{head} or \emph{top})
$\cosoc(M)$ is defined to be the maximal semisimple subquotient of $M$,
and is equal to $M/\rad(M)$.

A submodule $N$ of $M$ is called \emph{superfluous} if and only if
$K = M$ is the only
submodule of $M$ such that $N + K = M$. The dual notion is
a submodule $N$ of $M$ being \emph{essential}, which occurs if and
only if $K = \set{0}$ is the only submodule of $M$ such that
$N \cap K = \set{0}$. Then $\rad(M)$ can be characterized as
the sum of all superfluous submodules of $M$, while $\soc(M)$
can be characterized as the intersection of all essential submodules
of $M$.

The Jacobson radical of the endomorphism ring of self-projective
and self-injective modules can be characterized in terms of
superfluous and essential modules respectively.

\begin{lemma}[{\cite[Section~22]{wisbauer}}]
	\label{lem:self-proj-rad}
	Let $M$ be an $R$-module.
	\begin{enumerate}[(i)]
		\item If $M$ is self-projective, then
			\[
				\rad\paren{\End_R(M)} = \set{f \in \End_R(M) \,\middle|\, \im(f)
					\textrm{ is a superfluous submodule of } M}.
			\]
		\item If $M$ is self-injective, then
			\[
				\rad\paren{\End_R(M)} = \set{f \in \End_R(M) \,\middle|\, \ker(f)
					\textrm{ is an essential submodule of } M}.
			\]
	\end{enumerate}
\end{lemma}

\begin{proof}
	(i): Let $f \in \End_R(M)$ with superfluous image in $M$.
	Suppose $\End_R(M)f + A = \End_R(M)$ for an ideal $A$ of $\End_R(M)$.
	Then there is an $s \in \End_R(M)$ and $g \in A$ such that
	$sf + g = 1$. Then $M = Msf + Mg \subset \im(f) + Mg$, so $Mg = M$
	because $\im(f)$ is superfluous in $M$. Using the self-projectivity
	of $M$ on the diagram
	\[
		\begin{tikzcd}
			& M \arrow[d, twoheadrightarrow, "g"] \\
			M \arrow[r, "1"] \arrow[ur, dashed, "h"] & M,
		\end{tikzcd}
	\]
	we have an $h \in \End_R(M)$ such that $1 = hg \in A$, i.e.~$A =
	\End_R(M)$.
	So an $f \in \End_R(M)$ whose image is superfluous in $M$ actually
	satisfies the property that $\End_R(M) f$ is a superfluous submodule
	of $\End_R(M)$, and such an $f$ is therefore contained in
	$\rad\paren{\End_R(M)}$.
	
	Let $f \in \rad\paren{\End_R(M)}$. Suppose that $K$ is a submodule
	of $M$ such that $\im(f) + K = M$. Then the composition
	$M \xrightarrow{f} M \xrightarrow{p} M/K$ is an epimorphism, so we
	can use the self-projectivity of $M$ to make the following diagram
	commutative:
	\[
		\begin{tikzcd}
		  & M \arrow[d, "f"] \arrow[bend left=60, twoheadrightarrow, swap]{dd}{fp} \\
			& M \arrow[d, twoheadrightarrow, "p"] \\
			M \arrow[r, "p"] \arrow[uur, dashed, "h"] & M/K
		\end{tikzcd}
	\]
	From the commutativity of the diagram, $hfp = p$ so
	$(1-hf)p = 0$. Since $f \in \rad\paren{\End_R(M)}$,
	$1 - hf$ is invertible and hence $p = 0$.
	Therefore, $K = M$ and $\im(f)$ is superfluous in $M$.
	
	(ii): Take the dual of the arguments in the proof of (i).
\end{proof}

Using these characterizations, we can show that endomorphisms
of the cosocle of a self-projective module and endomorphisms
of the socle of a self-injective module lift to
endomorphisms on the original module.
\begin{thm}[{\cite[Section~22]{wisbauer}}]
	\label{thm:self-proj}
	Let $M$ be an $R$-module.
	
	\begin{enumerate}[(i)]
		\item	If $M$ is self-projective and
			$\rad(M)$ is a superfluous submodule of $M$, then
			\[
				\End_R(M)/\rad\paren{\End_R(M)} \cong \End_R\paren{M/\rad(M)}.
			\]
		\item	If $M$ is self-injective and
			$\soc(M)$ is an essential submodule of $M$, then
			\[
				\End_R(M)/\rad\paren{\End_R(M)} \cong \End_R\paren{\soc(M)}.
			\]
	\end{enumerate}
\end{thm}

\begin{proof}	
	(i): From the exact sequence
	\[
		\begin{tikzcd}
			0 \arrow[r] & \rad(M) \arrow[r] & M \arrow[r] &
				M/\rad(M) \arrow[r] & 0,
		\end{tikzcd}
	\]
	the following is also exact by the self-projectivity
	of $M$:
	\[
		\begin{tikzcd}
			0 \arrow[r] & \hom_R\paren{M, \rad(M)} \arrow[r] &
				\hom_R(M, M) \arrow[r] & \hom_R\paren{M, M/\rad(M)} \arrow[r]
				& 0.
		\end{tikzcd}
	\]
	Notice that $\hom_R\paren{M, M/\rad(M)} \cong \End_R\paren{M/\rad(M)}$
	since $\rad(M)$ is necessarily in the kernel of such an
	$R$-homomorphism.
	Since $\rad(M)$ is the sum of all superfluous submodules of $M$ and
	furthermore is itself a superfluous submodule of $M$ (by assumption),
	\[
		\hom_R\paren{M, \rad(M)} = \{f \in \End_R(M) \mid \im(f)
			\textrm{ is a superfluous submodule of } M\}.
	\]
	By Lemma~\ref{lem:self-proj-rad}(i), this is equal to
	$\rad\paren{\End_R(M)}$.
	
	Therefore, we have the exact sequence
	\[
		\begin{tikzcd}
			0 \arrow[r] & \rad\paren{\End_R(M)} \arrow[r] &
				\End_R(M) \arrow[r] & \End_R\paren{M/\rad(M)} \arrow[r]
				& 0,
		\end{tikzcd}
	\]
	so $\End_R\paren{M/\rad(M)} \cong \End_R(M)/\rad\paren{\End_R(M)}$.
	
	(ii): Consider an endomorphism $f: M \rightarrow M$ such that
	$\ker(f)$ is an essential submodule of $M$. Such an endomorphism
	factors through $M/\soc(M)$ since $\soc(M)$ is
	the intersection of all essential submodules of $M$.	
	By Lemma~\ref{lem:self-proj-rad}(ii),
	\[
		\rad\paren{\End_R(M)}	= \set{f \in \End_R(M) \,\middle|\,
			f \textrm{ is an essential submodule of } M }
			\subset \hom_R \paren{M/\soc(M), M}.
	\]
	Furthermore, any element $g \in \hom_R \paren{M/\soc(M), M}$ lifts
	to an endomorphism $g' \in \End_R(M)$ such that
	$\soc(M) \subset \ker(g')$. But $\soc(M)$ is an essential submodule
	of $M$ by assumption, so $\ker(g')$ is also an essential submodule
	of $M$. Then $g' \in \rad\paren{\End_R(M)}$ by
	Lemma~\ref{lem:self-proj-rad}(ii).
\end{proof}

\subsection{Exact functors}
\label{subsec:exact}
Finally, we briefly consider how exact functors can preserve or
transfer relative-projectivity. This will be used later for the
study of induction and restriction for representations
of finite groups in Section~\ref{subsec:ind-res}.

\begin{prop}
	\label{prop:rel-proj-functorial}
	Let $R_1$ and $R_2$ be two rings. Suppose that
	$\calf: \rmod{R_1} \rightarrow \rmod{R_2}$
	and $\calg: \rmod{R_2} \rightarrow \rmod{R_1}$
	are an exact adjoint pair $\calf \dashv \calg$.
	\begin{enumerate}[(i)]
		\item For $M \in \rmod{R_1}$ and
			$N \in \rmod{R_2}$,
			if $M$ is $\calg\paren{N}$-projective,
			then $\calf\paren{M}$ is $N$-projective.
		\item For $M \in \rmod{R_2}$ and
			$N \in \rmod{R_1}$,
			if $M$ is $\calf\paren{N}$-injective,
			then $\calg\paren{M}$ is $N$-injective.
	\end{enumerate}
\end{prop}

\begin{proof}
	(i): Suppose we have an exact sequence of $R_2$-modules
	\[
		\begin{tikzcd}
			0 \arrow[r] & L \arrow[r] & N \arrow[r] &
				K \arrow[r] & 0.
		\end{tikzcd}
	\]
	By the exactness of $\calg$, this gives
	the exact sequence
	\[
		\begin{tikzcd}
			0 \arrow[r] & \calg \paren{L} \arrow[r] &
				\calg \paren{N} \arrow[r] & \calg \paren{K} \arrow[r] & 0.
		\end{tikzcd}
	\]
	The $\calg \paren{N}$-projectivity of $M$ gives the exactness of
	\[
		\begin{tikzcd}
			0 \arrow[r] &
				\hom_{R_1} \paren{M, \calg \paren{L}} \arrow[r] &
				\hom_{R_1} \paren{M, \calg \paren{N}} \arrow[r] &
				\hom_{R_1} \paren{M, \calg \paren{K}} \arrow[r] & 0.
		\end{tikzcd}
	\]
	The adjunction between $\calf$ and $\calg$ implies the exactness of
	\[
		\begin{tikzcd}
			0 \arrow[r] &
				\hom_{R_2} \paren{\calf{\paren{M}}, L} \arrow[r] &
				\hom_{R_2} \paren{\calf{\paren{M}}, N} \arrow[r] &
				\hom_{R_2} \paren{\calf{\paren{M}}, K} \arrow[r] & 0.
		\end{tikzcd}
	\]
	Therefore, $\hom_{R_2} \paren{\calf{\paren{M}}, -}$ is exact.
	
	(ii): Suppose we have an exact sequence of $R_1$-modules
	\[
		\begin{tikzcd}
			0 \arrow[r] & L \arrow[r] & N \arrow[r] &
				K \arrow[r] & 0.
		\end{tikzcd}
	\]
	By the exactness of $\calf$, this gives
	the exact sequence
	\[
		\begin{tikzcd}
			0 \arrow[r] & \calf \paren{L} \arrow[r] &
				\calf \paren{N} \arrow[r] & \calf \paren{K} \arrow[r] & 0.
		\end{tikzcd}
	\]
	The $\calf \paren{N}$-projectivity of $M$ gives the exactness of
	\[
		\begin{tikzcd}
			0 \arrow[r] &
				\hom_{R_2} \paren{\calf \paren{L}, M} \arrow[r] &
				\hom_{R_2} \paren{\calf \paren{N}, M} \arrow[r] &
				\hom_{R_2} \paren{\calf \paren{K}, M} \arrow[r] & 0.
		\end{tikzcd}
	\]
	The adjunction between $\calf$ and $\calg$ implies the exactness of
	\[
		\begin{tikzcd}
			0 \arrow[r] &
				\hom_{R_1} \paren{L, \calg{\paren{M}}} \arrow[r] &
				\hom_{R_1} \paren{N, \calg{\paren{M}}} \arrow[r] &
				\hom_{R_1} \paren{K, \calg{\paren{M}}} \arrow[r] & 0.
		\end{tikzcd}
	\]
	Therefore, $\hom_{R_1} \paren{-, \calg{\paren{M}}}$ is exact.
\end{proof}


\section{General criteria for multiplicity-freeness}
\label{sec:multiplicity-freeness}

Using the endomorphism ring properties of self-projective
and self-injective modules
from Section~\ref{sec:rel-proj},
we can modify the proof of the classical Gelfand pair criterion
(Proposition~\ref{prop:classical}) to prove a
multiplicity-freeness theorem in any characteristic. Since the proof
still relies on Schur's lemma for endomorphism rings,
$R$ is an algebra over an algebraically closed base field $F$ in this
section.

\begin{thm}[Multiplicity-freeness, general version]
	\label{thm:multiplicity-free-2}
	Let $F$ be an algebraically closed field,
	$R$ be an algebra over $F$, and $M$ be a
	finitely-generated $R$-module.
	\begin{enumerate}[(i)]
		\item Suppose $M$ is a self-projective $R$-module. If
			$\End_R(M)$ is commutative, then
			\[
				\dim_F \hom_R\paren{M, N} \leq 1,
			\]
			for all simple $R$-modules $N$.
		\item Suppose $M$ is a self-injective and finitely-cogenerated
			$R$-module. If $\End_R(M)$ is commutative, then
			\[
				\dim_F \hom_R\paren{N, M} \leq 1,
			\]
			for all simple $R$-modules $N$.
	\end{enumerate}
\end{thm}

\begin{proof}
	If $M$ is a finitely-generated module, then its
	radical $\rad(M)$ is a superfluous submodule of $M$
	(\cite[Section~21.6]{wisbauer}).
	So if $M$ is both finitely-generated and self-projective,
	then it satisfies the conditions of Theorem~\ref{thm:self-proj}(i).
	Then by Theorem~\ref{thm:self-proj}(i),
	we obtain the isomorphism
	\[
		\End_R(M) / \rad\paren{\End_R(M)} \cong
		\End_R\paren{M / \rad (M)}.
	\]
	
	\noindent
	In particular, $M / \rad\paren{M}$ is semisimple
	and finitely-generated (since $M$ is finitely-generated).
	
	Similarly, if $M$ is a finitely-cogenerated module,
	then its socle $\soc(M)$ is an essential submodule of $M$.
	So if $M$ is both finitely-cogenerated and self-injective,
	then it satisfies the conditions of Theorem~\ref{thm:self-proj}(ii).
	Then by Theorem~\ref{thm:self-proj}(ii),
	we obtain the isomorphism
	\[
		\End_R(M) / \rad\paren{\End_R(M)} \cong
		\End_R\paren{\soc (M)}.
	\]
	\noindent
	In particular, $\soc(M)$ is semisimple and
	finitely-generated (since $M$ is
	finitely-cogenerated \cite[Section~21.3]{wisbauer}).
	
	Let $L = M / \rad\paren{M}$ if $M$ is self-projective,
	and otherwise let $L = \soc(M)$ if $M$ is
	finitely-cogenerated and self-injective. In both cases,
	$L$ is semisimple and finitely-generated, so
	we may write $L = \bigoplus d_i N_i$
	as a direct sum of distinct simple $R$-modules $N_i$ with
	positive integer multiplicities $d_i$.
	In both cases, we have that
	\[
		\End_R(M) / \rad\paren{\End_R(M)} \cong \End_R\paren{L} \cong
		\bigoplus \mat (d_i, D_i),
	\]
	for division algebras $D_i$.
	Since	$\End_R(M)$ is commutative,
	so is $\bigoplus \mat (d_i, F)$.
	Therefore, the $D_i$ are actually fields,
	all $d_i \leq 1$, and
	the semisimple $R$-module
	$L$ is multiplicity-free.
	Then if $L = M/\rad(M)$,
	\[
		\dim_F \hom_R\paren{M, N}
			= \dim_F \hom_R\paren{M/\rad(M), N}
			\leq 1,
	\]
	for all simple $R$-modules $N$.
	If $L = \soc(M)$, then
	\[
		\dim_F \hom_R\paren{N, M}
			= \dim_F \hom_R\paren{N, \soc(M)}
			\leq 1
	\]
	for all simple $R$-modules $N$.
\end{proof}

\begin{rem}
	In the self-injective case of
	Theorem~\ref{thm:multiplicity-free-2}(ii), the $R$-module $M$ is
	required to be \textit{both} finitely-generated (for Schur's lemma)
	and finitely-cogenerated (for $\soc(M)$ to be an essential submodule
	of $M$).
	If $R$ is Artinian (e.g. any group ring of a finite group),
	then we may ignore the finite-cogeneration condition
	because all finitely-generated
	$R$-modules are finitely-cogenerated.
\end{rem}

For applications of Theorem~\ref{thm:multiplicity-free-2} to group
representations, $R$ is the group ring $F[G]$ of a group $G$.
The formulation of Theorem~\ref{thm:multiplicity-free} follows
directly from Theorem~\ref{thm:multiplicity-free-2}(ii) by setting
$M$ to be a group representation $\rho$ of $G$.

For applications of Theorem~\ref{thm:multiplicity-free-2} to group
representations,
we specialize to the group ring $R = F[G]$
of a finitely-generated group $G$ with $M = \rho$
a finite-dimensional representation.
Theorem~\ref{thm:multiplicity-free-2}(ii) can then be directly
reformulated for group representations, but we can strengthen the
result in this setting by removing the
finite-cogeneration condition through an application of duality to
Theorem~\ref{thm:multiplicity-free-2}(i).

\begin{thm}[Multiplicity-freeness, group representation version]
	\label{thm:multiplicity-free-3}
	Let $F$ be an algebraically closed field,
	$G$ be a finitely-generated group,
	and $\rho$ be a finite-dimensional representation of $G$.
	For all irreducible representations $\pi$ of $G$,
	\[
		\dim_F \hom_G\paren{\pi, \rho} \leq 1,
	\]
	if both of the following conditions are satisfied:
	\begin{enumerate}[(i)]
		\item $\End_G\paren{\rho}$ is commutative;
		\item $\rho$ is self-injective.
	\end{enumerate}
\end{thm}

\begin{proof}
	Since $\rho$ is self-injective,
	the representation $\widetilde{\rho}$ is
	self-projective by duality.
	Furthermore, $\End_G\big(\widetilde{\rho}\big) \cong
	\End_G\big(\rho\big)$ since $\rho$
	is finite-dimensional. 
	Then by Theorem~\ref{thm:multiplicity-free-2}(i)
	with $R = F[G]$ and $M = \widetilde{\rho}$,
	\[
		\dim_F \hom_{F[G]}\paren{\widetilde{\rho}, N}
		\leq 1,
	\]
	for all simple $F[G]$-modules $N$.
	Viewing the simple $F[G]$-modules $N$ as duals of
	irreducible representations of $G$
	(using the fact that the dual of a finite-dimensional representation
	is irreducible
	if and only if the original representation is irreducible),
	this means that
	\[
		\dim_F \hom_{F[G]}\paren{\widetilde{\rho}, \widetilde{\pi}} \leq 1,
	\]
	for all irreducible representations $\pi$ of $G$.
	By the natural isomorphism
	\[
		\hom_{F[G]}\paren{\widetilde{\rho}, \widetilde{\pi}}
		\cong \hom_{F[G]}\paren{\pi, \rho},
	\]
	we conclude that
	\[
		\dim_F \hom_{F[G]}\paren{\pi, \rho} \leq 1,
	\]
	for all irreducible representations $\pi$ of $G$ over $F$.
\end{proof}

Unlike Gelfand's criterion in the characteristic zero setting,
our proof of Theorem~\ref{thm:multiplicity-free-2} (and therefore
Theorem~\ref{thm:multiplicity-free-3})
does not provide a converse theorem because in the isomorphism,
\[
	\End_R(M) / \rad\paren{\End_R(M)} \cong
		\bigoplus \mat (d_i, F),
\]
the commutativity of the right-hand side
does not necessarily imply the commutativity of $\End_R(M)$.
In fact, the converse is false as illustrated
by the following non-example.

\begin{non-example}
	\label{nonex:non-commutative}
	Let $F$ be a field of characteristic $p$,
	$G$ be a finite non-abelian $p$-group, $R = F[G]$, and
	$M = F[G]$.
	Then $\End_R(M) = F[G]$, which is non-commutative.
	But the only irreducible representation of $G$ over $F$
	is the trivial representation, so $M$ still satisfies the
	multiplicity-one property despite failing the commutativity
	condition of Theorem~\ref{thm:multiplicity-free-2}.
\end{non-example}

The conditions of Theorem~\ref{thm:multiplicity-free-2} are
also somewhat necessary, as one can otherwise construct
a representation with commutative Hecke algebra but
higher multiplicity if it is not self-injective.
\begin{non-example}
	\label{nonex:higher-multiplicity}
	Let $\pi$ be an irreducible
	representation of a group $G$ over a field $F$ with two non-split
	non-isomorphic extensions $\sigma_1$ and $\sigma_2$:
	\[
		\begin{tikzcd}
			0 \arrow[r] & \pi \arrow[r] & \sigma_1 \arrow[r] & \tau_1 \arrow[r] & 0, \\
			0 \arrow[r] & \pi \arrow[r] & \sigma_2 \arrow[r] & \tau_2 \arrow[r] & 0,
		\end{tikzcd}
	\]
	where $\tau_1$ and $\tau_2$ are irreducible representations
	of $G$ over $F$ different from $\pi$. Such extensions exist,
	for example, for groups in positive characteristic
	or even for $p$-adic groups in characteristic $0$.
	Then there is the extension $\rho := \sigma_1 \oplus \sigma_2$ of
	$\pi \oplus \pi$:
	
	\[
		\begin{tikzcd}
			0 \arrow[r] & \pi \oplus \pi \arrow[r] &
				\rho = \sigma_1 \oplus \sigma_2 \arrow[r] &
				\tau_1 \oplus \tau_2 \arrow[r] & 0.
		\end{tikzcd}
	\]
	Observe that for $i = 1$ or $2$,
	any $\phi \in \End_G(\rho)$
	sends $\sigma_i$ to itself,
	and $\End(\sigma_i)$ sends $\pi$ to itself
	since $\pi$ is different from $\tau_i$.
	Also,
	$\End_G(\rho) =
	\End_G(\sigma_1) \oplus \End_G(\sigma_2)$
	since $\sigma_1 \not\cong \sigma_2$.
	Then $\End_G(\rho) \cong F \oplus F$
	is commutative, but $\pi$ has multiplicity
	greater than $1$ in $\rho$.
\end{non-example}


\section{Finite and compact multiplicity-free triples}
\label{sec:finite-compact}

Theorem~\ref{thm:multiplicity-free-3} is quite general,
but when specializing to $R = F[G]$ for a finite group $G$,
we may use induction and restriction properties of self-projectivity
and self-injectivity
to remove condition (ii). We also consider the specialization
to smooth representations of totally disconnected compact groups.

\subsection{Induction and restriction}
\label{subsec:ind-res}

Let $G$ be a group and $H$ be a subgroup of $G$. Assume that
the left and right adjoints of the restriction functor
$\res_H^G$ exist.
Define induction $\ind_H^G$ from $\rep_F(H)$ to $\rep_F(G)$ to be the
left adjoint of
restriction $\res_H^G$ from $\rep_F(G)$ to $\rep_F(H)$,
and define coinduction $\coind_H^G$ to be the right
adjoint of $\res_H^G$.

\begin{rem}
	When $G$ is discrete, $\res_H^G$
	has both left-adjoints and right-adjoints (cf. \cite{hristova}).
	For our settings of interest, induction and coinduction are often
	given concretely as $\ind_H^G (M) = F[G] \otimes_{F[H]} M$ and
	$\coind_H^G (M) = \hom_{F[H]} (F[G], M)$.
	In the context of representations of locally profinite groups,
	coinduction $\coind_H^G$ is commonly called ``induction''
	and denoted $\ind_H^G$
	while induction $\ind_H^G$ is commonly called ``compact induction''
	and denoted $\cind_H^G$ or $\lind_H^G$
	(cf. \cite[I.1.5]{vigneras}).
\end{rem}

When specializing the exact adjoint pair
$\calf \dashv \calg$ to
induction-restriction $\ind_H^G \dashv \res_H^G$
and restriction-coinduction $\res_H^G \dashv \coind_H^G$,
Proposition~\ref{prop:rel-proj-functorial} has
the following consequence.

\begin{cor}
	\label{cor:rel-proj-indres}
	Let $F$ be an algebraically closed field,
	$G$ be a group, and $H$ be a subgroup of $G$.
	\begin{enumerate}[(i)]
		\item Let $M$ be an $F[G]$-module and let $N$ be an
			$F[H]$-module.
			If $\coind_H^G$ is an exact functor
			and $M$ is $\coind_H^G \paren{N}$-projective,
			then $\res_{H}^G \paren{M}$ is an $N$-projective
			$F[H]$-module.
		\item Let $M$ be an $F[H]$-module and let $N$ be an
			$F[G]$-module.	
			If $\ind_H^G$ is an exact functor
			and $M$ is $\res_{H}^G \paren{N}$-projective,
			then $\ind_H^G (M)$ is an $N$-projective
			$F[G]$-module.
		\item Let $M$ be an $F[G]$-module and let $N$ be an
			$F[H]$-module.
			If $\ind_H^G$ is an exact functor
			and $M$ is $\ind_H^G \paren{N}$-injective,
			then $\res_{H}^G \paren{M}$ is an $N$-injective
			$F[H]$-module.
		\item Let $M$ be an $F[H]$-module and let $N$ be an
			$F[G]$-module.	
			If $\coind_H^G$ is an exact functor
			and $M$ is $\res_{H}^G \paren{N}$-injective,
			then $\coind_H^G (M)$ is an $N$-injective
			$F[G]$-module.
	\end{enumerate}
\end{cor}

\begin{proof}
	\hfill
	
	(i): Proposition~\ref{prop:rel-proj-functorial}(i) with
	$R_1 = F[G]$, $R_2 = F[H]$,
	$\calf = \res_H^G$, and $\calg = \coind_H^G$.
	
	(ii): Proposition~\ref{prop:rel-proj-functorial}(i) with
	$R_1 = F[H]$, $R_2 = F[G]$,
	$\calf = \ind_H^G$, and $\calg = \res_H^G$.
	
	(iii): Proposition~\ref{prop:rel-proj-functorial}(ii) with
	$R_1 = F[H]$, $R_2 = F[G]$,
	$\calf = \ind_H^G$, and $\calg = \res_H^G$.
	
	(iv): Proposition~\ref{prop:rel-proj-functorial}(ii) with
	$R_1 = F[G]$, $R_2 = F[H]$,
	$\calf = \res_H^G$, and $\calg = \coind_H^G$.
\end{proof}

If induction and coinduction are
exact functors, then we can
generate some examples of relatively-projective
representations using Corollary~\ref{cor:rel-proj-indres}.
We will assume that $G$ is finite for simplicity, but
exactness of induction and coinduction
actually holds in greater generality
(cf. \cite[Section~8.16]{jantzen} for finite algebraic groups
and \cite[I.1.5]{vigneras} for locally profinite
groups). Induction and coinduction are equal
for finite groups, so we will generally only 
say ``induction'' and use $\ind_H^G$ 
in the finite group case.

\begin{lem}
	\label{lem:self-proj-ind}
	Let $F$ be an algebraically closed field,
	$G$ be a finite group, and $H$ be a subgroup of $G$.
	\begin{enumerate}[(i)]
		\item If $\rho$ is an irreducible representation
			of $G$, then $\rho$ is a self-projective $F[G]$-module.
		\item If $\rho$ is an irreducible representation
			of $G$, then $\rho$ is a self-injective $F[G]$-module.
		\item If $\rho$ is a self-projective $F[H]$-module,
			then the induced representation $\ind_H^G(\rho)$ is a
			self-projective $F[G]$-module.
		\item If $\rho$ is a self-injective $F[H]$-module,
			then the induced representation $\ind_H^G(\rho)$ is a
			self-injective $F[G]$-module.
	\end{enumerate}
\end{lem}

\begin{proof}
	\hfill
	
	(i)	A simple $R$-module $M$ is self-projective
	because a surjective morphism
	$g: M \twoheadrightarrow K$ 
	is either zero or an isomorphism.
	If $g$ is zero, then $f$ is also zero
	so any endomorphism $h$ suffices.
	If $g$ is an isomorphism, then $h = f \circ g^{-1}$
	is the desired lifting of any morphism $f: M \rightarrow K$.
	
	(ii) This follows by duality to (i).
	
	(iii) By Mackey's restriction formula,
	\[
		\res_{H}^{G} \paren{\ind_H^G (\rho)} \cong
		\bigoplus_{[x] \in H\backslash G/H} \ind_{H_s}^H \rho_s,
	\]
	where $H_s := sHs^{-1} \cap H$ and $\rho_s$ is the representation
	of $H_s$ defined by $\rho_s(x) := \res_{H_s}^H(\rho) (s^{-1}xs)$.
	Note that this is a finite direct sum of elements of
	$C^p(\rho)$ since
	\[
		\ind_{H_s}^H \paren{\rho_s} = \bigoplus_{[y] \in H_s \backslash H} y^{-1} (\rho_s),
	\]
	where $y^{-1} (\rho_s)$ is the image of the action of a
	representative $y^{-1} \in H$ of $[y]^{-1}$
	on the subspace $[1] \otimes \rho_s \subset k[H] \otimes_{k[H_s]} \rho_s
	= \ind_{H_s}^H \rho_s$. 
	
	Since $G$ is finite, Proposition~\ref{prop:rel-proj-properties}
	implies that $\rho$ is a
	$\res_{H}^{G} \paren{\ind_H^G (\rho)}$-projective
	module in two different ways:
	by Proposition~\ref{prop:rel-proj-properties}(i) because
	$[G:H]$ is finite, and by
	Proposition~\ref{prop:rel-proj-properties}(ii)
	because $\dim \rho$ is finite.
	Consequently, $\ind_H^G(\rho)$ is a self-projective $F[G]$-module
	by Corollary~\ref{cor:rel-proj-indres}(ii) with
	$M = \rho$ and $N = \ind_H^G(\rho)$.
	
	(iv) This follows by duality to (iii) for coinduction $\coind_H^G$
	and the equality $\ind_H^G = \coind_H^G$.
\end{proof}

\subsection{Commutative Hecke algebra criteria}
Again assuming $G$ to be finite for simplicity,
we may use Lemma~\ref{lem:self-proj-ind} to remove condition
(ii) of Theorem~\ref{thm:multiplicity-free-3}.
Recall that the Hecke algebra $\calh(G, H, \triv_H, \C)$
is a convolution algebra that is
isomorphic to $\End_G(\ind_H^G(\triv_H))$.

\begin{thm}
	\label{thm:gelfand-triple-2}
	Let $F$ be an algebraically closed field, $G$ be a finite group,
	$H$ be a subgroup of $G$, and $\eta$ be an
	irreducible representation of $H$.
	If $\calh(G, H, \eta, F)$ is commutative, then $(G, H, \eta)$ is a
	multiplicity-free triple.
\end{thm}

\begin{proof}
	Let $\rho = \ind_H^G\paren{\eta}$. Since
	$\dim \eta < \infty$,
	$\rho$ is finite-dimensional.
	Furthermore, $\rho$
	is a self-injective $F[G]$-module
	by irreducibility and Lemma~\ref{lem:self-proj-ind},
	so condition (ii) of
	Theorem~\ref{thm:multiplicity-free-3} is satisfied.
\end{proof}

\begin{rem}
	Only one of the two finiteness assumptions on
	$[G:H]$ and $\dim \eta$ is used
	to satisfy condition (ii) of Theorem~\ref{thm:multiplicity-free-3}.
\end{rem}

A common method of proving that a Hecke algebra is commutative
is with the criterion known as Gelfand's trick or
Gelfand's lemma, which works in any characteristic.
\begin{lem}[Gelfand's trick]
	\label{lem:gelfand-trick}
	Let $F$ be an algebraically closed field, $G$ be a finite group,
	$H$ be a subgroup of $G$, and $\eta$ be an irreducible
	representation of $H$.
	If there is an anti-involution $\iota$ such that
	$f(\iota(g)) = f(g)$ for all $f \in \calh(G, H, \eta, F)$
	and all $g \in G$,
	then $\calh(G, H, \eta, F)$ is commutative.
\end{lem}
For Gelfand pairs (i.e. $\eta = \triv_H$), this
condition is equivalent to $\iota$ preserving
all double cosets of $H$.
Together with Theorem~\ref{thm:gelfand-triple-2},
this extends Gelfand's trick to multiplicity-free triples
over fields of arbitrary characteristic.

\begin{cor}
	\label{cor:gelfand-trick}
	Let $F$ be an algebraically closed field, $G$ be a finite group,
	$H$ be a subgroup of $G$, and
	$\eta$ be an irreducible representation of $H$.
	If there is an anti-involution $\iota$ that preserves
	all double cosets of $H$,
	then $(G, H, \eta)$ is a multiplicity-free triple over $F$.
\end{cor}

In the case that $F$ is the algebraic closure
of a finite field, another consequence of
Theorem~\ref{thm:gelfand-triple-2} together with Brauer theory
is that it is not even necessary
to use Gelfand's trick over $F$ if
$(G, H, \eta)$ is already known to be a multiplicity-free triple
over $\C$. We briefly outline this
process for finite groups following Serre \cite[Part~III]{serre}
and Curtis--Reiner \cite[Section~16C]{curtis-reiner}
(for other groups, cf.
\cite{serre-1968, vigneras-1989, vigneras, zhang}).

Given a complex representation $(\rho, V)$ of a finite group $G$,
there is a notion of reduction modulo $\ell$ for any
prime $\ell$,
\[
	\rho \mapsto \overline{\rho},
\]
yielding a representation $\overline{\rho}$ of $G$ over
a finite extension of $\F_\ell$.
Since $G$ is finite,
the representation $(\rho, V)$ can be realized over
the algebraic integers $\calo_E$ of a number field
$E \hookrightarrow \overline{\Q}$; using the embedding 
$\Q \hookrightarrow \Q_\ell$, $\rho$ can then be viewed as
a representation $(\rho_K, V)$ over the ring of integers
$\sco_K$ of a finite extension $K$ of $\Q_\ell$.
Let $\mfm_K$ denote the maximal ideal of $\sco_K$ and
let $k$ denote the residue field $\sco_K/\mfm_K \sco_K$.
Pick a lattice
$\Lambda \subset V$ (i.e. a finitely-generated
$\calo_K$-submodule $\Lambda$ of $V$ that
generates $V$ as a $K$-module).
Such a lattice exists and
can be taken to be $G$-stable (i.e. $g \Lambda \subset
\Lambda$ for all $g \in G$) by replacing $\Lambda$ by
$\bigoplus_{g \in G} g \Lambda$.
Define the reduction,
\[
	\overline{\Lambda} := \Lambda / \mfm_K \Lambda;
\]
this is a $k[G]$-module for a finite extension
$k$ of $\F_\ell$.
It is known due to Brauer--Nesbitt \cite{brauer-nesbitt}
(cf. \cite[Theorem~32]{serre} and
\cite[Proposition~16.16]{curtis-reiner})
that this reduction modulo $\ell$ is independent
of the choice of lattice $\Lambda$
in the sense that every such reduction has
the same composition factors,
so we denote the reduction by
$\overline{\rho}$.
Take the algebraic closure of the finite field $k$ to view
$\overline{\rho}$ as a representation of $G$
over $F = \overline{\F}_\ell$.

As an immediate consequence of Theorem~\ref{thm:gelfand-triple-2}, it
is entirely sufficient to know that $(G, H, \eta)$ is a
multiplicity-free triple over the complex numbers
to prove that $(G, H, \overline{\eta})$ is a multiplicity-free
triple over such an $F$. We state the following result
for a character $\eta$ for simplicity,
but a similar statement should also hold
for general irreducible complex representations
$\eta$.

\begin{corollary}
	\label{cor:Fq-2}
	Let $F = \overline{\F}_\ell$ of any positive characteristic $\ell$,
	$G$ be a finite group, $H$ be a subgroup of $G$, and
	$\eta$ be a one-dimensional representation of $H$ over $\C$.
	If $(G, H, \eta)$ is a multiplicity-free triple over $\C$,
	then $(G, H, \overline{\eta})$ is also a multiplicity-free triple
	over $F$.
\end{corollary}

\begin{proof}
	Over $\C$, the converse of
	Theorem~\ref{thm:gelfand-triple-2} is also true by
	Proposition~\ref{prop:classical}. Therefore,
	the Hecke algebra $\calh(G, H, \eta, \C)$ is commutative.
	$G$ is a finite group, so
	$\eta$ is defined over a finite extension $E$ of $\Q$
	with commutative $\calh(G, H, \eta, E)$.
	This commutativity naturally extends to the Hecke algebra
	$\calh(G, H, \eta_K, \sco_K)$ over the ring of integers
	of an $\ell$-adic field $K$.
	
	Since $\End_\C(\eta) \cong \C$ and
	$\End_F(\eta) \cong F$,
	the Hecke algebra $\calh(G, H, \eta, \C)$
	(resp. $\calh(G, H, \eta, F)$)
	can be described as the space of continuous
	$\C$-valued (resp. $F$-valued) functions $\Delta$ on $G$ such that
	$\Delta(h_2 g h_1) = \eta(h_2) \cdot \Delta(g) \cdot \eta(h_1)$
	(resp. with $\overline{\eta}$),
	which has a basis over $\C$ (resp. $F$)
	corresponding to indicator functions
	on the finite double coset space $H \bs G / H$.
	By the definitions of $E$, $K$ and $\overline{\eta}$,
	the Hecke algebras over
	$E$, $\sco_K$, and $k$ have similar descriptions
	with coefficients restricted from $\C$ and $F$.
	In particular, there is a homomorphism of Hecke algebras,
	\[
		\begin{tikzcd}[row sep=tiny]
			\calh(G, H, \eta_K, \sco_K) \arrow[r]
				& \calh(G, H, \overline{\eta}, k) \\
			\displaystyle\sum_{\xi \in H \bs G / H} a_\xi \cdot \Delta_\xi \arrow[r, mapsto]
				& \displaystyle\sum_{\xi \in H \bs G / H} \overline{a}_\xi \cdot \overline{\Delta}_\xi,
		\end{tikzcd}
	\]
	where $\Delta_\xi: G \rightarrow \sco_K$
	is the basis element corresponding to
	the indicator function of $\xi$,
	$a_\xi \in \sco_K$, and $\overline{a}_\xi, \overline{\Delta}_\xi$
	are their reductions modulo $\mfm_K$.
	Since $a_\xi$ is arbitrary and
	elements of $k$ lift to $\sco_K$, the homomorphism is
	surjective and therefore the commutativity of
	$\calh(G, H, \eta_K, \sco_K)$ implies
	the commutativity of $\calh(G, H, \overline{\eta}, k)$.
	
	Finally, the commutativity of $\calh(G, H, \overline{\eta}, k)$
	implies the commutativity of
	$\calh(G, H, \overline{\eta}, F)$
	by the flatness of commutative Hecke algebras.
	Hence $(G, H, \overline{\eta})$ is also a multiplicity-free triple
	over $F$ by Theorem~\ref{thm:gelfand-triple-2}.
\end{proof}

\begin{remark}
	For higher-dimensional $\eta$,
	one must take care to assume that $\overline{\eta}$ is irreducible
	over $\overline{\F}_\ell$
	in order to use Theorem~\ref{thm:gelfand-triple-2}.
	For instance,
	only the Steinberg representation remains irreducible after
	equal-characteristic reduction for many finite groups of Lie type
	(cf. Tiep--Zalesskii \cite{tiep-zalesskii1, tiep-zalesskii2,
	tiep-zalesskii3}).
	Reduction modulo $\ell$ is only well-defined
	up to composition factors, so Corollary~\ref{cor:Fq-2} applies
	to all reductions modulo $\ell$ (i.e. all choices of lattices)
	simultaneously (and all reductions modulo $\ell$ are
	either simultaneously irreducible or simultaneously reducible).
\end{remark}

\subsection{Compact groups}
\label{subsec:compact}
A totally disconnected compact group $G$ is profinite,
so there is an inverse system of finite groups $\set{G_i}_{i \in \N}$
with compatible homomorphisms $f_i^j: G_j \rightarrow G_i$
for $i \leq j$ of which $G$ is a projective limit,
\[
	G = \varprojlim_{i \in \N} G_i.
\]
The kernel of each map $G \rightarrow G_i$ is an open normal subgroup
$U_i$, and any open subgroup $H$ of $G$ contains
one of these open normal subgroups $U_i$.

A representation of a totally disconnected
group is called \textit{smooth} if the
stabilizer subgroup of any vector of the representation is open.
In particular, smooth representations of profinite groups factor
through finite quotients. In this way,
the results for representations of finite groups can be directly
applied to smooth representations of compact groups,
giving a compact version of
Corollary~\ref{cor:Fq-2}.
Here, we consider coinduction of smooth representations,
with $\coind_H^G(M)$ given by $\varinjlim_{i \in \N}
\hom_{F[H]}(F[G_i], M)$.

\begin{cor}
	\label{cor:Fq-compact}
	Let $F = \overline{\F}_\ell$ of any positive characteristic $\ell$,
	$G$ be a totally disconnected compact group,
	$H$ be a closed subgroup of $G$, and
	$\eta$ be a one-dimensional smooth representation of $H$ over $\C$.
	If $(G, H, \eta)$ is a multiplicity-free triple over $\C$,
	then $(G, H, \overline{\eta})$ is also a multiplicity-free triple
	over $F$.
\end{cor}

\begin{proof}
	Since $\eta$ is smooth irreducible and $H$ is a closed subgroup
	of a profinite $G$,
	$\eta$ is finite-dimensional and there is an open
	normal subgroup $U_j$ of $G$ such that $U_j' := U_j \cap H$
	fixes $\eta$, i.e. $\eta$ arises from
	a representation $\eta^{U_j'}$ of a finite quotient
	$H_j := H/U_j'$. Similarly, $\overline{\eta}$ is
	finite-dimensional and factors through a representation
	$\overline{\eta}^{U_j'}$ of a finite
	quotient $H_j$.
	Since $\eta$ is invariant by $U_j'$, $\overline{\eta}$ is
	also invariant by $U_j'$ and
	$\overline{\eta^{U_j'}} \cong \overline{\eta}^{U_j'}$.
	
	The commutativity of $\calh(G, H, \eta, \C)$ implies
	the commutativity of $\calh(G_j, H_j, \eta^{U_j'}, \C)$,
	which in turn gives the commutativity of
	$\calh(G_j, H_j, \overline{\eta^{U_j'}}, \overline{\F}_\ell)
	= \calh(G_j, H_j, \overline{\eta}^{U_j'}, \overline{\F}_\ell)$ by
	Corollary~\ref{cor:Fq-2}. This argument works for all
	$i \geq j$, so each
	$\calh(G_i, H_i, \overline{\eta}, \overline{\F}_\ell)$
	is commutative for $i \geq j$.
	The representation $\coind_H^G(\overline{\eta})$
	is a direct limit over $i \geq j$
	of $U_i'$-invariant representations
	$\coind_{H_i}^{G_i}(\overline{\eta}^{U_i'})$,
	which are each finite-dimensional.
	The Hecke algebra
	$\calh(G, H, \overline{\eta}, \overline{\F}_\ell)$ is a
	projective limit of the commutative
	$\calh(G_i, H_i, \overline{\eta}^{U_i'}, \overline{\F}_\ell)$;
	therefore,
	$\calh(G, H, \overline{\eta}, \overline{\F}_\ell)$ is commutative.
\end{proof}

\begin{remark}
	As with Corollary \ref{cor:Fq-2},
	this compact version should
	also hold if $\eta$ is an irreducible smooth representation
	of any finite dimension over $\C$ such that
	its modulo-$\ell$ reduction $\overline{\eta}$
	is irreducible over $F$.
	Furthermore, the other results of Section~\ref{subsec:ind-res}
	for finite groups should also hold for profinite groups.
	One can use the factorization of smooth representations
	through finite quotients and apply the theory of finite groups as
	done immediately above,
	or modify the proofs using suitable profinite analogues
	(e.g. the equivalence of induction and coinduction for
	locally profinite groups when $H \backslash G$ is compact
	(cf. \cite[I.1.5.2]{vigneras})
	and Mackey decomposition for locally profinite groups (cf.
	\cite[I.1.5.5]{vigneras}).	
\end{remark}


\section{Applications for finite and compact groups}
\label{sec:examples}
When $F$ is the complex numbers, there are many well-known examples of
Gelfand pairs and tools for finding multiplicity-free triples,
such as the generalized Bump--Ginzburg criterion
(\cite[Theorem~3.2]{csst}),
so it is not difficult to find cases to use
Corollary~\ref{cor:Fq} and Corollary~\ref{cor:Fq-compact}.
In this section, we highlight a few applications of
Theorem~\ref{thm:multiplicity-free-2} for
finite groups and compact groups.

\begin{rem}
	There are also many useful applications of
	multiplicity-free triples for locally compact groups. For instance,
	the Iwasawa decomposition of Gelfand pairs of locally compact
	groups with compact subgroups has consequences for the
	non-commutativity of Hecke algebras of Coxeter groups
	(cf. \cite{monod}). It may be more tractable to apply
	Theorem~\ref{thm:multiplicity-free-2}
	to the theory of automorphic representations, where
	the condition of finite generation may not be an obstruction for
	admissible representations of totally disconnected locally compact
	groups. Multiplicity-one theorems for $p$-adic groups,
	such as the uniqueness of
	Whittaker models by Jacquet--Langlands \cite{jacquet-langlands} and
	Shalika \cite{shalika} as well as its generalizations
	(e.g. Vign\'{e}ras \cite[III.1.11]{vigneras}), are key tools
	in the characteristic $0$ and characteristic $\ell \neq p$ theory.
\end{rem}

\subsection{Gelfand--Graev and \texorpdfstring{$\gl(n, q)$}{GL(n, q)}}
For $G = \gl_n\paren{F_{q}}$, there are many well-known
examples of Gelfand pairs, such as the infinite families
(cf. Bannai--Tanaka \cite{bannai-tanaka}):
\begin{itemize}
	\item $\paren{\gl_n\paren{\F_{{q}^2}}, \gl_n\paren{\F_{q}}}$;
	\item $\paren{\gl_n\paren{\F_{{q}^2}}, \gu_n\paren{\F_{q}}}$;
	\item $\paren{\gl_{2n}\paren{\F_{q}}, \Sp_{2n}\paren{\F_{q}}}$;
	\item $\paren{\gl_{2n}\paren{\F_{q}}, \gl_n\paren{\F_{{q}^2}}}$.
\end{itemize}
We recall a particularly important
twisted Gelfand pair that was first given by
Gelfand--Graev \cite{gelfand-graev}
(cf. \cite[Theorem~14.6.3]{ceccherini},
Piatetski--Shapiro \cite[Proposition~10.3]{ps}, and
\cite[Section~4.1]{bump}).

Let $G = \gl_2\paren{\F_{q}}$ for a prime power $q$ and consider the
unipotent subgroup,
\begin{align*}
	U &= \set{\begin{pmatrix}1 & x \\ 0 & 1\end{pmatrix} \in G \mid x \in \F_{q}}.
\end{align*}

\noindent
The result of Gelfand--Graev,
reinterpreted in the language of multiplicity-free triples,
is that $(G, U, \eta)$ is
a multiplicity-free triple over $\C$ for any nontrivial character
$\eta$ of $U$ by virtue of the commutativity of the Hecke algebra
$\calh(G, U, \eta)$. This multiplicity-freeness result
(for complex Gelfand--Graev representations) is the key ingredient in
the proof that there are $\frac{q(q-1)}{2}$ cuspidal representations
of $G$ and that every cuspidal representation of $G$ has dimension
$q - 1$.

The multiplicity-freeness result of Gelfand--Graev generalizes to
$G = \gl_n\paren{\F_{q}}$ for $n > 2$. In this case, there is the
unipotent subgroup $U \leq G$ of the form
\begin{align*}
	U &= \set{
		\begin{pmatrix}1 & x_{12} & x_{13} & \ldots & x_{1k} \\
									   & 1      & x_{23} & \ldots & x_{2k} \\
										 &        & 1      & \ddots & \vdots \\
										 &        &        & \ddots & \vdots \\
										 &        &        &        & 1
		\end{pmatrix} \in G \,\middle|\, x_{ij} \in \F_{q}},
\end{align*}
and given a nontrivial character $\psi$ of $\F_{q}$, a character
$\eta_\psi$ can be defined on $U$ via
\[
	\eta_\psi\paren{(x_{ij})} := \psi\paren{\sum_{i=1}^{n-1} x_{i,i+1}}.
\]
In this setting, $\paren{G, U, \eta_\psi}$ is a multiplicity-free
triple over $\C$ for any nontrivial character $\psi$ of $\F_{q}$.
By Corollary~\ref{cor:Fq},
the multiplicity-freeness result extends
to modular representations over $F$ the algebraic closure of any finite
field (even if $\ell := \Char F$ divides $q$ or the order of $G$).

\begin{cor}
	\label{cor:Gelfand-Graev-2}
	Let $F = \overline{\F}_\ell$ of any positive characteristic $\ell$
	and let $q$ be any prime power.
	For any nontrivial character $\psi$ of $\F_{q}$,
	$(G, U, \eta_\psi)$ is
	a multiplicity-free triple over $F$.
	In other words, $(G, U)$ is a $\eta_\psi$-twisted Gelfand pair
	over $F$.
\end{cor}

\begin{remark}
	\label{rem:steinberg}
	In the equal characteristic ($\ell \mid q$) setting,
	the $\psi = \triv_U$ case
	can be deduced from classical works
	that $(G, U, \eta_\psi)$ is a multiplicity-free triple.
	A theorem of Steinberg \cite[Theorem~44(b)]{steinberglectures}
	(largely based on a result of Curtis \cite[Theorem~4.1]{curtis1965}
	and later extended from finite Chevalley groups to finite groups
	with a split ($B$, $N$)-pair / Tits building by
	Richen \cite[Theorem~3.9]{richen},
	cf. Curtis \cite[Theorem~4.3]{curtis1970})
	says that every irreducible representation of $G$
	over $F$ has a unique fixed vector (up to scalars)
	by $U$, where the unique fixed vector corresponds to the
	highest weight. This implies that $(G, U)$ is a Gelfand
	pair in equal characteristic without
	the twist by $\eta_\psi$ considered by
	Corollary~\ref{cor:Gelfand-Graev-2}.
	Note that this is the only case in the equal characteristic setting,
	since there are no nontrivial additive characters
	$\psi$ when $\ell \mid q$.
\end{remark}

Corollary~\ref{cor:Gelfand-Graev-2} directly recovers the uniqueness of
Whittaker models over all characteristics.

\begin{cor}
	\label{cor:whittaker-2}
	Let $F = \overline{\F}_\ell$ of any positive characteristic $\ell$
	and let $q$ be any prime power.
	An irreducible representation of $\gl_n(\F_{q})$
	over $F$ has at most one Whittaker model
	for a choice of a
	nontrivial additive character of $\F_{q}$.
\end{cor}

There are also known cases of complex multiplicity-free triples for
more general representations $\eta$, which therefore extend to
modular multiplicity-free triples by Corollary~\ref{cor:Fq}. For
instance,
Ceccherini-Silberstein--Scarabotti--Tolli \cite[Section~6]{csst}
proved that for an odd prime power $q$,
$\paren{\gl_2\paren{\F_{q}}, C, \eta}$ is a multiplicity-free triple
over $\C$, where $C$ is the Cartan subgroup
\begin{align*}
	C &= \set{\begin{pmatrix}a & b\zeta \\ b & a\end{pmatrix} \in G
		\mid a, b \in \F_q \backslash \set{0}, \zeta \textrm{ a generator of } \F_q^*} \cong \F_{{q}^2}
\end{align*}
and $\eta$ is an multiplicative character of $\F_{{q}^2}$.

\begin{remark}
	Given the connections between the Gelfand--Graev representation
	and cuspidal representations,
	this example hints at interesting questions about how modular
	multiplicity-freeness results could work in the setting of
	cuspidal representations. In particular,
	Baruch--Rallis \cite{baruch-rallis} defined a notion of
	(super)cuspidal Gelfand pairs $(G, H)$ that naturally extends to
	triples $(G, H, \eta)$. It would be useful to develop
	general modular techniques in this direction, where
	for over algebraically closed fields of characteristic coprime to
	odd $q$,
	S\'{e}cherre \cite[Corollary~2.16]{secherre} and
	Zou \cite[Theorem~1.2]{zou-unitary}
	showed that $\paren{\gl_{2n}(\F_{q}),
	\gl_{n}(\F_q) \times \gl_n(\F_{q})}$ is a cuspidal Gelfand pair
	and
	$\paren{\gl_{n}\paren{\F_{q}}, \paren{\gl_{n}(\F_{q})}^\tau}$,
	where $\tau$ is a unitary involution,
	is a supercuspidal Gelfand pair.
\end{remark}

\subsection{Trilinear forms of quaternion division algebras}
For the quaternion division algebra $D_k$
over a local field $k$, its multiplicative group modulo center
$D_k^* / Z$ is compact (cf. \cite[Section~6]{carayol}).
Since any irreducible smooth
representation of $D_k^*$ has a central character, any such
representation factors through a finite quotient after a twist.

One application of Corollary~\ref{cor:Fq} is that
the multiplicity-one theorem for trilinear
forms on complex irreducible smooth representations of quaternion
division algebras over local fields by
Prasad \cite[Theorem 1.1]{prasad-thesis}
can be extended to all characteristics.
\begin{cor}
	\label{cor:qda-2}
	Let $F = \overline{\F}_\ell$ of any positive characteristic $\ell$,
	$k$ be any local field, and $D_k$ be the quaternion division
	algebra over $k$.
	If $\pi_1, \pi_2, \pi_3$ are three irreducible smooth
	representations of $\qda$ over $F$ for some prime
	power $q$, then there exists at most one
	(up to isomorphism) non-zero $\qda$-invariant linear form
	on $\pi_1 \otimes \pi_2 \otimes \pi_3$ over $F$.
\end{cor}

\begin{proof}
	Define the tensor product representation
	$\pi := \pi_1 \otimes \pi_2 \otimes \pi_3$ of
	$G := \qda \times \qda \times \qda$. Let $H := \Delta \qda$ be the
	diagonal subgroup of $G$.
	The number of non-zero $\qda$-invariant linear forms
	on $\pi_1 \otimes \pi_2 \otimes \pi_3$ over $F$
	is given by
	\[
		\dim_{F} \hom_{H}
			\paren{\res_{H}^G \paren{\pi}, \triv_{H}},
	\]
	which is at most one if
	$(G, H)$ is a Gelfand pair.
	
	Over $\C$, the anti-involution $\iota$ of
	$G$ defined by
	\[
		\iota(x, y, z) := \paren{\tr(x) - x, \tr(y) - y, \tr(z) - z}
	\]
	preserves $H$-cosets (\cite[Proposition 3.3]{prasad-thesis}).
	Therefore,
	$(G, H)$ is a Gelfand pair
	over $\C$. By Corollary~\ref{cor:Fq-compact},
	it is also a Gelfand
	pair over $F$.
\end{proof}
The original result by Prasad over the complex numbers
was one of the ingredients for the proof by
Harris--Kudla \cite{harris-kudla-1991,
harris-kudla-2004} of the Jacquet conjecture on the
non-vanishing of the central value of triple product L-functions.
With the further understanding of modular multiplicity-free triples of
locally compact groups, one could consider
modular versions of the work of Prasad and Harris--Kudla to describe
modular properties of the central value of triple product L-functions.


\section{Multiplicity-free restrictions}
\label{sec:restrictions}

Let $F$ be an algebraically closed field and
let $\rho$ be an irreducible $F[S_n]$ representation,
where $S_n$ is the symmetric group on $n$ letters.
Consider $S_{n-1}$ as the subgroup of $S_n$ consisting
of the permutations of the first $n-1$ letters.
If the characteristic $F$ is zero then the restriction
$\res^{S_n}_{S_{n-1}}(\rho)$ is completely reducible and
multiplicity-free, with its composition factors described
by the standard branching theorem.
If the characteristic of $F$ is non-zero, then
the multiplicity of a composition factor can be
arbitrarily large \cite[Corollary~3.3]{jantzen-seitz}.

Kleschev \cite{kleschev-1, kleschev-2, kleschev-4}
proved the conjectures of Benson \cite{benson}
(for characteristic $2$) and
Jantzen--Seitz \cite{jantzen-seitz} (for odd characteristic)
surrounding the question of when $\res^{S_n}_{S_{n-1}}(\rho)$ is irreducible
for $p>0$. One can broaden these conjectures to the following question.
\begin{question}[{\cite[Problem 1]{kleschev-morotti-tiep}}]
	\label{q:kleschev}
	Let $F$ be an algebraically closed field.
	Classify the triples $(G, H, \rho)$, where $H$ is a subgroup of $G$ and
	$\rho$ is a representation of $G$ over $F$ of dimension greater than
	$1$ such that the restriction $\res^{G}_{H}(\rho)$ is irreducible.
\end{question}
For maximal subgroups of classical algebraic groups over $F = \C$,
the study of this question goes back to Dynkin \cite{dynkin}.
For $G = S_n$ or $A_n$,
this question is completely answered over fields of characteristic $0$
by Saxl \cite{saxl}, over fields of characteristic greater than $3$
by Brundan--Kleschev \cite{brundan-kleschev} and
Kleschev--Sheth \cite{kleschev-sheth}, and
over fields of characteristic $2$ and $3$ by 
Kleschev--Morotti--Tiep \cite{kleschev-morotti-tiep}.

For general fields, the resolution of Question~\ref{q:kleschev}
has applications to several problems in representation theory,
such as the Aschbacher--Scott program on understanding
maximal subgroups of finite classical groups
(cf. \cite[Section~1]{kleschev-morotti-tiep}).

We can consider an broader question, allowing the restriction
$\res^{G}_{H}(\rho)$ to be reducible but asking when
it is composed of unique
irreducible representations.
\begin{question}
	\label{q:mult-free}
	Let $F$ be an algebraically closed field.
	Classify the triples $(G, H, \rho)$, where $H$ is a subgroup of $G$ and
	$\rho$ is a representation of $G$ over $F$ of dimension greater than
	$1$ such that the restriction $\res^{G}_{H}(\rho)$ is multiplicity-free.
\end{question}

For irreducible representations of classical algebraic groups and
reductive groups over $F = \C$,
this problem is closely related to the classical works
on $H$-varieties and invariant theory
of Kac \cite[Theorem~3]{kac}, Weyl \cite{weyl}, and
Howe \cite[Chapters 3-4]{howe}. Much of this body of work
(including extensions by Brion \cite{brion}, Benson--Ratcliff
\cite{benson-ratcliff},
and Leahy \cite{leahy}, as well as
the classification for simple algebraic groups
of Kr\"{a}mer \cite{kramer})
can be interpreted as answering Question~\ref{q:mult-free},
under the more restrictive condition that
$\res^{G}_{H}(\rho)$ is multiplicity-free
for all irreducible representations $\rho$ of $G$
(cf. \cite[Section~5.7, Section~12.2]{goodman-wallach} and
\cite[Chapter~1]{lst2}).

Recently, Liebeck--Seitz--Testerman \cite{lst1, lst2} gave a classification
of irreducible representations $\rho$ with multiplicity-free
restrictions when $G$ and $H$ are simple algebraic groups of type $A$
over an algebraically closed field of characteristic $0$
and $H$ is an irreducible subgroup of $G$.
Their result uses the work of Stembridge \cite{stembridge}, who 
answered Question~\ref{q:mult-free}
for tensor products $\rho_1 \otimes \rho_2$ of
irreducible representations of a simple algebraic group $H$ of type
$A$, with $G = \gl(\rho_1) \times \gl(\rho_2)$,
over an algebraically closed field of characteristic $0$.

Question~\ref{q:mult-free} is also related to the notion of strong
Gelfand pairs. For finite or compact groups,
the pair $(G, H)$ is a \emph{strong Gelfand pair}
if
\[
	\dim_F \hom_H \paren{\restr{\rho}{H}, \eta} \leq 1
\]
for every irreducible representation $\rho$ of $G$ and every
irreducible representation $\eta$ of $H$.
This is then equivalent to
$(G \times H, \Delta H)$ being a Gelfand pair
(here $\Delta H$ is the diagonal of $H$ in $H \times H$)
and is therefore characterized by the commutativity of a
Hecke algebra.

A naive application of Theorem~\ref{thm:multiplicity-free-3} and
Corollary~\ref{cor:rel-proj-indres}(iii) to Question~\ref{q:mult-free}
is the following corollary.
\begin{cor}
	\label{cor:restriction}
	Let $F$ be an algebraically closed field,
	$H$ be a subgroup of a discrete finitely generated group $G$ such that
	$\ind_H^G$ is an exact functor, and $\rho$ be a finite-dimensional
	representation of $G$.
	If $\rho$ is $\ind_H^G(\res_H^G(\rho))$-injective and
	$\End_H(\res_H^G(\rho))$ is commutative, then
	$\res^{G}_{H}(\rho)$ is multiplicity-free.
\end{cor}

\begin{proof}
	By Corollary~\ref{cor:rel-proj-indres}(iii)
	with $M = \rho$ and $N = \res_H^G(\rho)$,
	$\res_H^G(\rho)$ is self-injective. Then by
	Theorem~\ref{thm:multiplicity-free-3}, we have
	the multiplicity-freeness of $\res^{G}_{H}(\rho)$.
\end{proof}


\section*{Acknowledgements}
I am grateful to my advisor, Michael Harris, for first
suggesting a project on trilinear forms of modular representations of
quaternion division algebras and $\gl_2$ which ultimately
led to this work.
I am also grateful to Dipendra Prasad for many suggestions
in this paper,
such as Non-examples \ref{nonex:non-commutative} and
\ref{nonex:higher-multiplicity} and Remark \ref{rem:steinberg}.
I would also like to thank
Benedict Gross, Alexander Kleschev,
Chao Li, Gilbert Moss, Vincent S\'{e}cherre, and
Marie-France Vign\'{e}ras for helpful comments and suggestions.
The anonymous referee deserves special thanks for feedback
that led to important corrections and clarifications.

This material is based on work
supported by the National Science Foundation
under Grant No. DGE-1644869 and Grant No. DMS-2303280.


\bibliography{bibliography}{}
\bibliographystyle{amsalpha}

\end{document}